%% file: ArMe2_arxiv.tex
\documentclass[11pt]{article}

\input{paquetes}
\input{comandos}

\begin{document}

\title{A posteriori error analysis for the Navier-Stokes equations with non-smooth data}

\author[1]{María Gabriela Armentano}
\author[1]{Mauricio Mendiluce}
\affil[1]{Departamento de Matemática, Facultad de Ciencias Exactas y Naturales, Universidad de Buenos Aires, IMAS - Conicet, Buenos Aires 1428, Argentina}

\date{}

\maketitle

\begin{abstract}
We study the stationary Navier–Stokes equations with Dirichlet boundary data in $L^2$, a setting in which the limited regularity of the solution prevents the direct application of standard a posteriori error estimation techniques. To address this issue, we introduce a regularized formulation that yields a well-posed approximation of the original problem and admits a conforming finite element discretization. Using Taylor–Hood  $P_2P_1$ elements, we construct a residual-based a posteriori error estimator and establish its reliability and efficiency under suitable smallness assumptions on the data. We derive computable upper and lower bounds in an appropriate norm that relate the estimator to the error between the exact solution of the original Navier–Stokes problem and its finite element approximation, showing that the estimator accurately reflects the finite element error. These results provide a rigorous foundation for the analysis and implementation of adaptive finite element methods for incompressible flows with low-regularity Dirichlet boundary data.
\end{abstract}



\section{Introduction}\label{sec1}

The stationary Navier–Stokes equations play a central role in the mathematical modeling of incompressible viscous flows and have been extensively studied from both the theoretical and numerical viewpoints (see, for example,  \cite{ArMe,CGO2021,ern2012,girault1986,gunzburger1983,temam1977,volker} and the references therein).
In particular, finite element methods combined with adaptive mesh refinement strategies have become a fundamental tool for the efficient numerical approximation of incompressible flow problems. Indeed, for Stokes equations a posteriori error analysis is well established (see, for instance, \cite{AM2014,dari1995,duran2020,hamouda2017,song2014})  while a broad literature has also been developed for the Navier--Stokes equations, where several adaptive finite element techniques and corresponding error analyses have been proposed \cite{araya2014,arnica1997,camano2022,leng2023,verfurth1993,zhang2013}. 
The design and analysis of reliable a posteriori error estimators constitute a key ingredient in this framework, since they provide quantitative information on the quality of the numerical approximation and guide adaptive refinement procedures (see, for example, the classical book by Verfürth  \cite{verfurth2013}).
Most available a posteriori analyses for the Navier–Stokes equations rely on regularity assumptions on the data and the exact solution, especially on the boundary conditions. However, in several relevant situations the Dirichlet boundary data may possess only low regularity, for instance belonging to $L^2$. In this setting, standard weak formulations and classical residual-based techniques cannot be directly applied.
A regularized formulation for this low-regularity setting was previously analyzed in \cite{ArMe}, where suitable a priori estimates relating the original and regularized solutions were established in an appropriate framework. In the present work we develop a residual-type a posteriori error estimator for the regularized problem and prove its reliability and efficiency under suitable smallness assumptions on the data. More precisely, using Taylor--Hood $P_2P_1$ finite elements, we derive upper and lower bounds for the discretization error in $L^4$ norm allowing a comparison between the original solution and the discrete approximation.

We begin by introducing the notation used throughout the paper. Let $s\ge 0$, $1\le p < \infty$, and let $D\subset \mathbb{R}^2$ be an open bounded domain. We denote by $W^{s,p}(D)$ the usual Sobolev space on $D$, endowed with norm $\|\cdot\|_{W^{s,p}(D)}$ and seminorm $|\cdot|_{W^{s,p}(D)}$ (see, e.g., \cite{Adams1975}). The space $W_0^{s,p}(D)$ is defined as the closure of $C_0^\infty(D)$ in $W^{s,p}(D)$.
In the particular case $p=2$, we use the standard notation
$H^s(D)=W^{s,2}(D)$, $H_0^s(D)=W_0^{s,2}(D)$ and we also define
$$\|v\|_{H^{-1}(D)}=\sup_{w\in H_0^1(D), w\neq 0}
\frac{\int_D v \, w}{|w|_{H^1(D)}}.
$$

Throughout the paper, boldface letters denote vector-valued functions and the corresponding vector-valued functional spaces.

Given $\O\subset\matR^2$ a bounded domain with a polygonal boundary. We consider the following problem

\begin{equation}\label{NS_dato_g}\left\{
    \begin{array}{rl}
         -\Delta\bu +\Rey(\bu\cdot\nabla)\bu + \nabla p = \bf &   \quad \text{in } \O  \\
         \div \bu = 0 & \quad \text{in } \O\\
         \bu = \bg & \quad  \text{on }\Gamma=\partial\O
    \end{array}\right.
\end{equation}
where $\Rey$ is the Reynolds number, $\bf\in\bH^{-1}(\O)$ and $\bg\in \bL^{2}(\Gamma)$ such that

\begin{equation}\label{cond_compatibilidad}
    \int_\Gamma\bg\cdot\bn = 0.    
\end{equation}
   
In \cite{marusic2000} the existence of a very weak solution in $L^4(\O)$ is shown assuming that $\|\bg\|_{L^2({\Gamma})}$ is sufficiently small. 

Now, let $\tilde{\bg} \in \bH^{1/2}(\Gamma)$ a regularization of $\bg$ such that $\int_\Gamma \tilde{\bg}\cdot\bn = 0$ 

We consider the regular problem of \eqref{NS_dato_g}:

\begin{equation}\label{NS_dato_ge} \left\{
    \begin{array}{rl}
         -\Delta\bv +\Rey(\bv\cdot\nabla)\bv + \nabla q = \bf &   \quad \text{in } \O  \\
         \div \bv = 0 & \quad \text{in } \O\\
         \bv = \tilde{\bg} & \quad \text{on } \Gamma=\partial\O
    \end{array}\right.
\end{equation}

For this new boundary data we will consider the weak formulation: find $(\bv,q)\in\bH^1(\O)\times L^2_0(\O)$ tal que $\bv|_\Gamma=\tilde{\bg}$ and 

\begin{equation}\label{NSweak_regular_data} \left\{
	 \begin{array}{rl}
    \int_\O\nabla\bv:\nabla\bw+\Rey\int_\O(\bv\cdot\nabla)\bv\cdot\bw-\int_\O q\div\bw  &= \int_\O\bf\cdot\bw \quad \forall\bw\in\bH^1_0(\O) \\
		\int_\O \tau\div\bv  &=0 \quad \forall \tau\in L^2_0(\O)
	\end{array} \right.
\end{equation}

In the classical literature \cite{girault1986,temam1977}, the existence and uniqueness of solution to \eqref{NSweak_regular_data} are established 
under suitable assumptions on the data. One may also refer to \cite[Theorem 6.36]{ern2012} for the homogeneous case and extend the argument to the non-homogeneous setting. In \cite{ArMe}, it is shown, under an appropriate assumption on $\tilde{\bg}$, that there exists a constant $C>0$ such that
\begin{equation}\label{error:u-v}
 \|\bu-\bv\|_{L^4(\O)}\leq C\|\bg-\tilde{\bg}\|_{L^2(\Gamma)}.\end{equation}

This estimate allows us to control the difference between the solutions of the original and regularized problems in the $L^4$-norm. Relying on this result, we show in the present work that the proposed a posteriori error estimator is equivalent to the $L^4$-error between the regularized solution and its finite element approximation. Combining these two results, we finally obtain an equivalence relation between the estimator and the error associated with the original Navier–Stokes problem.

The rest of the paper is organized as follows.
In Section \ref{sec2}, we present preliminary results, including extension results.  In Section \ref{fem} we establish a priori estimates for the regularized problem and its finite element approximation. In Section \ref{sec:aposteriori}, we introduce a posteriori error estimator and analyze its reliability and efficiency under suitable smallness
assumptions on the data. Finally, in Section \ref{sec:example}, we present numerical experiments for the cavity flow problem, where the performance of the proposed estimator is assessed using the average  marking strategy and compared with uniform mesh refinement.

\section{Preliminary results}\label{sec2}

In this section we introduce some preliminary results that we will use throughout the article.

\begin{lemma}
    Given $\bv\in\bH^1(\O)$ there exists $C_I>0$ such that 
    \begin{equation}\label{inmersion}
        \|\bv\|_{L^4(\O)}\leq C_I \|\bv\|_{H^1(\O)}  
    \end{equation}
\end{lemma}

\begin{proof}
    See \cite{Adams1975,Grisvard1985}.
    \qed
\end{proof}

\begin{lemma}\label{prop_oprerador_no_lineal}
    Let $\bu,\bv,\bw\in\bH^1(\O)$. The following properties hold:
    \begin{itemize}
        \item[(a)] $\left|\int_\O(\bu\cdot\nabla)\bv\cdot\bw\right|\leq \|\bu\|_{L^4(\O)}\|\nabla\bv\|_{L^2(\O)}\|\bw\|_{L^4(\O)}$. Furthermore, there exists $\matC_a = \matC_a(\O)>0$ such that $$\left|\int_\O(\bu\cdot\nabla)\bv\cdot\bw\right|\leq \matC_a\|\bu\|_{H^1(\O)} \|\bv\|_{H^1(\O)} \|\bw\|_{H^1(\O)}.$$
        \item[(b)] If $\div\bu=0$ and either $\bv|_\Gamma =0$ or $\bw|_\Gamma= 0$, then $$\int_\O (\bu\cdot\nabla)\bv\cdot\bw = -\int_\O (\bu\cdot\nabla)\bw\cdot\bv.$$

        \item[(c)] If $\div\bu=0$, and  $\bv|_\Gamma =0$ then $\int_\O(\bu\cdot\nabla)\bv\cdot\bv=0$
    \end{itemize}
\end{lemma}
\begin{proof}
    See \cite[Lemma 2.1]{ArMe}.
    \qed
\end{proof}

The following lemma provides a divergence-free extension of functions\\ in $\bH^{1/2}(\Gamma)$.

\begin{lemma}\label{extension_dato_dirichlet}
Given $\bxi\in \bH^{1/2}(\Gamma)$ satisfying
$\int_{\Gamma} \bxi \cdot \bn =0$, for any $\varepsilon>0$
there exists a function $\bpsi_{\varepsilon} \in \bH^1(\Omega)$ such that
\[
\div \bpsi_{\varepsilon}=0, \qquad \bpsi_\varepsilon|_{\Gamma}=\bxi,
\]
and
\[
\left| \int_\Omega (\bw\cdot\nabla)\bpsi_\varepsilon\cdot\bw \right|
\leq \varepsilon \|\bw\|_{H^1(\Omega)}^2,
\qquad \forall \bw\in \bH_0^1(\Omega),\ \div \bw=0.
\]
Moreover, there exists a constant $C_\varepsilon>0$ such that
\[
\|\bpsi_\varepsilon\|_{H^1(\Omega)}
\le C_\varepsilon \|\bxi\|_{H^{1/2}(\Gamma)}.
\]
\end{lemma}

\begin{proof}
We refer to \cite[Lemma 2.3, Chapter IV]{girault1986} for the construction of the extension. Moreover, the construction yields the stated estimate
\[
\|\bpsi_\varepsilon\|_{H^1(\Omega)}
\le C_\varepsilon \|\bxi\|_{H^{1/2}(\Gamma)}.
\]
\qed
\end{proof}

Now, we present an a priori estimate for $\bv$, the solution to \eqref{NSweak_regular_data}., which will be used in the proofs of the reliability and efficiency of the a posteriori error estimator.

\begin{lemma}\label{estimacion_||v||}
Let $\bv$ be the solution to \eqref{NSweak_regular_data} and let $0 < \eps < \frac{C_p}{Re}$, with  $C_p$  the Poincar\'e constant defined in \eqref{Poincare}, then 
    \begin{equation}\label{a_priorio_estim_v}
        \|\bv\|_{H^1(\O)} \leq \frac{\|\tilde{\bf}\|_{H^{-1}(\O)} + (1+C_p-\Rey\,\eps)\|\bpsi\|_{H^1(\O)}}{ C_p - \Rey\,\eps},
    \end{equation}
    where $\bpsi= \bpsi_\eps$ is an extension of $\tilde{\bg}$ satisfying Lemma \ref{extension_dato_dirichlet}   and $\tilde{\bf} = \bf -\Rey\, (\bpsi\cdot\nabla)\bpsi$.     
\end{lemma}

\begin{remark}\label{equivalent_problem}
For the case of homogeneous Dirichlet boundary conditions (e.g. $\bg=0$), following the ideas in \cite[Chapter IV, Section 1]{girault1986}, in particular Theorem 1.4, the analysis is carried out in the space
\[
\bV = \{\boldsymbol{\phi} \in \mathbf{H}^1(\Omega) : \operatorname{div}\boldsymbol{\phi} = 0\}.
\]

If $\bv \in \bV$ satisfies
\[
\int_\Omega \nabla \bv : \nabla \bw
+ \mathrm{Re} \int_\Omega (\bv\cdot\nabla)\bv \cdot \bw
= \int_\Omega \mathbf{f}\cdot \bw,
\qquad \forall\, \bw \in \bV,
\]
then, by the inf-sup condition associated with the bilinear form\\
$b(\bw,q) = \int_\Omega q\,\operatorname{div}\bw$, it can be concluded that there exists
$(\bv,q)\in \bH^1_0(\Omega)\times L^2_0(\Omega)$ solving \eqref{NSweak_regular_data}.
\end{remark} 

\begin{proof}
Following the strategy employed, for example, in \cite{temam1977} to establish existence and uniqueness results for non-homogeneous problems, we define $$\bv^0 = \bv-\bpsi$$
Thus, $\bv^0\in\bH^1_0(\O)$ and $\div\bv^0 = 0$. From Remark \ref{equivalent_problem}, it suffices to consider the following problem:

\begin{align*}
\int_\Omega \nabla (\bv^0+\bpsi) : \nabla \bw
+ \Rey \int_\Omega \big( ((\bv^0+\bpsi)\cdot\nabla)(\bv^0+\bpsi) \big)\cdot \bw
= \int_\Omega \bf\cdot \bw.
\end{align*}

Taking $\bw = \bv^0$ we get
\begin{align*}
\|\nabla \bv^0\|_{L^2(\O)}^2
+ \int_\Omega \nabla \bpsi : \nabla \bv^0
+ \Rey \int_\Omega \Big[
(\bv^0\cdot\nabla)\bv^0\cdot \bv^0
+ (\bpsi\cdot\nabla)\bv^0\cdot \bv^0 \\
+ (\bv^0\cdot\nabla)\bpsi\cdot \bv^0
+ (\bpsi\cdot\nabla)\bpsi\cdot \bv^0
\Big]
= \int_\Omega \bf\cdot \bv^0.
\end{align*}

From property (c) of Lemma \ref{prop_oprerador_no_lineal} we have that  $\int_\Omega (\bv^0\cdot\nabla)\,\bv^0\cdot \bv^0 =0$ and\\  $\int_\Omega(\bpsi\cdot\nabla)\,\bv^0 \cdot \bv^0 =0$ so, 
 
\begin{align*}
\|\nabla \bv^0\|_{L^2(\O)}^2
 &= \int_\Omega \bf\cdot \bv^0
- \int_\Omega \nabla \bpsi : \nabla \bv^0
- \Rey \int_\Omega (\bv^0\cdot\nabla)\bpsi\cdot \bv^0
- \Rey \int_\Omega (\bpsi\cdot\nabla)\bpsi\cdot \bv^0.\\
 &= \int_\Omega \tilde{\bf}\cdot \bv^0
- \int_\Omega \nabla \bpsi : \nabla \bv^0
- \Rey \int_\Omega (\bv^0\cdot\nabla)\bpsi\cdot \bv^0.
\end{align*}

Using Holder inequality, item (a) of Lemma \ref{prop_oprerador_no_lineal} and Lemma \ref{extension_dato_dirichlet} It follows that
\begin{align*}
\|\nabla \bv^0\|_{L^2(\O)}^2 &\leq \|\tilde{\bf}\|_{H^{-1}(\Omega)} \|\bv^0\|_{H^1(\Omega)}
+ \|\nabla \bpsi\|_{L^2(\O)} \|\nabla \bv^0\|_{L^2(\O)}
+ \Rey\,\eps \|\bv^0\|_{H^1(\O)}^2
\end{align*}

Now, since $\bv^0 \in \bH^1_0(\Omega)$, by the Poincaré inequality, we can affirm that there exists a constant $C_p$ such that
\begin{equation}\label{Poincare}
C_p \|\bv^0\|^2_{H^1(\Omega)} \leq \|\nabla \bv^0\|^2_{L^2(\Omega)}.
\end{equation}

So,
\begin{align*}
\left( C_p -  \Rey\,\eps \right)
\|\bv^0\|_{H^1(\O)}
\leq \|\tilde{\bf}\|_{H^{-1}(\O)} + \|\bpsi\|_{H^1(\O)}.
\end{align*}
and since  $C_p-\Rey\,\eps>0$ we obtain

$$\|\bv^0\|_{H^1(\O)}\leq \frac{\|\tilde{\bf}\|_{H^{-1}(\O)} + \|\bpsi\|_{H^1(\O)}}{ C_p - \Rey\,\eps}.$$

Finally, 

\begin{align*}
    \|\bv\|_{H^1(\O)} &\leq \|\bv^0\|_{H^1(\O)}+\|\bpsi\|_{H^1(\O)}\leq \frac{\|\tilde{\bf}\|_{H^{-1}(\O)} + (1+C_p-\Rey\,\eps)\|\bpsi\|_{H^1(\O)}}{ C_p - \Rey\,\eps}
\end{align*}
and the proof is complete.
\qed
\end{proof}

As a direct consequence of Lemma \ref{extension_dato_dirichlet}  and Lemma \ref{estimacion_||v||} we have the following a priori estimate.
\begin{corollary} 
Let $\bv$ be the solution of \eqref{NSweak_regular_data} and let $0 < \eps < \frac{C_p}{Re}$, then
\begin{equation} \label{cota_v}
    \|\bv\|_{H^1(\O)} \leq \frac{\|\bf\|_{H^{-1}(\O)} + \Rey\, C_aC_\eps^2 \| \tilde{\bg} \|^2_{H^{1/2}(\Gamma)} + (1+C_p-\Rey\,\eps)C_\eps\|\tilde{\bg}\|_{H^{1/2}(\Gamma)}}{ C_p - \Rey\,\eps}
\end{equation} 
where  $C_a,C_p$ and $C_{\eps}$ 
are as in the previous lemmas.
\end{corollary}

\section{Finite Element Aproximation}\label{fem}

Let $\Omega$ be a bounded domain with polygonal boundary and 
let  $\left\{{\mathcal T}_h\right\}$ be a family of triangulations of $\O$ such that
any two triangles in ${\mathcal T}_h$ share at most a vertex or an
edge. Let $h$ stand for the mesh-size; namely $h=\max_{T\in{\mathcal
T}_h}h_T$, with $h_T$ being the diameter of the triangle $T$. We
assume that the family of triangulations $\left\{{\mathcal T}_h\right\}$ satisfies the shape-regularity condition, i.e., there exists a constant $\sigma>0$
such that $\frac{h_T}{\rho_T} \le \sigma$, where $\rho_T$ is the
diameter of the largest circle contained in $T$.
For any \( T \in \mathcal{T}_h \), we denote by $\mathcal{E}_T$ the set of edges of $T$,  \( \mathcal{E}_h = \bigcup_{ T \in  \mathcal{T}_h }  \mathcal{E}_T \) and we decompose
\[
\mathcal{E}_h = \mathcal{E}_{h,\Omega} \cup \mathcal{E}_{h,\Gamma}, \quad \mathcal{E}_{h,\Omega} \cap \mathcal{E}_{h,\Gamma} = \emptyset,
\]
where \( \mathcal{E}_{h,\Gamma} \) denotes the set of all edges contained in \( \Gamma \) and $\mathcal{E}_{h,\Omega} = \mathcal{E}_h \setminus \mathcal{E}_{h,\Gamma} $. 
Given \( \ell \in \mathcal{E}_h \), we denote by \( \omega_\ell \) the union of the triangles in \( \mathcal{T}_h \) having \( \ell \) as an edge. Similarly, \( \omega_T \), for \( T \in \mathcal{T}_h \), is the union of all triangles sharing an edge with \( T \).

We consider the stable spaces $\bV_h$ and $Q_h$ taking the \textit{Taylor-Hood} elements, such that $\bV_h=\bH^1(\Omega)\cap(P_2(\matT_h))^2$, $\bV_{h,0}=\bH^1_0(\Omega)\cap(P_2(\matT_h))^2$ and $Q_h = L^2_0(\O)\cap P_1(\matT_h)$.  Where $P_k(\matT_h)$ denotes the 
piecewise polynomials of degree less than or equal to $k$ on the mesh $\matT_h$.

Throughout the paper, $C$ denotes a generic constant, not necessarily the same at each occurrence, that is independent of the parameter $h$.
Given two quantities $A$ and $B$ the notation $A\lesssim B$ means that $A\le CB$. We also denote by $A \sim B$ when $A\lesssim B$ and $B\lesssim A$.

From now on, we set $\tilde{\bg}=\bg_h$, where $\bg_h$ denotes a regular approximation of $\bg$ such that $\bg_h$ is the trace of a function in $\bV_h$. Possible choices for $\bg_h$ include the $L^2(\Gamma)$-projection, the Carstensen interpolant, and the Lagrange interpolant (all of them  were considered in \cite{duran2020} where the authors analyze the numerical approximation of the Stokes problem with not smooth data). However, in general, these constructions do not guaranty the compatibility condition $\int_\Gamma \bg_h \cdot \bn = 0$. 
In \cite{apel2026}, the authors argue that attaining a sufficiently high approximation order for the discretization error $\bg-\bg_h$ is more important for convergence than enforcing the discrete compatibility condition itself. Nevertheless, they also provide a procedure to enforce this condition. 
Throughout our analysis, we assume that $\bg_h$ is chosen so that the compatibility condition holds.

We consider the problem of finding $(\bu_h,p_h)\in\bV_h\times Q_h$ such that $\bu_h|_\Gamma=\bg_h$ and

\begin{equation}\label{NS_fem} \left\{
	 \begin{array}{rll}
        \int_\O\nabla\bu_h:\nabla\bw_h+\Rey\, \int_\O (\bu_h\cdot\nabla)\bu_h\cdot\bw_h-\int_\O p_h\div\bw_h  &= \int_\O\bf\bw_h \quad &\forall\bw_h\in\bV_{h,0} \\
		\int_\O q_h\div\bu_h  &=0 \quad &\forall q_h\in Q_h
	\end{array} \right.
\end{equation}

In \cite{gunzburger1983}, existence and uniqueness are proved for this problem. 
On the other hand, an a priori error estimate was established in \cite{ArMe}, where it is shown that 
\begin{equation}\label{estima-error-u-uh}
\| \bu - \bu_h \|_{L^4(\Omega)} \lesssim \| \bg - \bg_h \|_{L^2(\Gamma)}.
\end{equation} 

The following lemma, which provides an extension of $\bg_h$, is a fundamental tool for obtaining an a priori estimate for the solution $\bu_h$ of \eqref{NS_fem}.

\begin{lemma}\label{extension_gh}
Let $\bg_h \in \bH^{1/2}(\Gamma)$ be the trace of a function on $\bV_h$  satisfying the compatible condition $\int_{\Gamma} \bg_h \cdot \bn =0$. Then, there exists 
an extension function $\bG_h\in \bV_h $ such that:

$$  \bG_h|_{\Gamma} = \bg_h, \qquad  \div \bG_h = 0, $$

and $\|\bG_h\|_{\bH^1(\Omega)} \leq C_g \|\bg_h\|_{H^{1/2}(\Gamma)}$, with $C_g$ independent of $h$.
\end{lemma}

\begin{proof}
Let $ \bE_h \in \bV_h$ be an extension of $\bg_h$  such that $\bE_h|_{\Gamma} = \bg_h $ and $\|\bE_h\|_{\bH^1(\Omega)} \lesssim \|\bg_h\|_{H^{1/2}(\Gamma)}$ as constructed, for instance, in \cite{duran2020}. We emphasize that, in general, $\bE_h$ is not divergence-free.

Let us define
\[
q_h := \div \bE_h 
\]
Since $(\bV_{h,0}, Q_h)$ is inf-sup stable,  and $\div(\bV_{h,0}) = Q_h $,
the discrete divergence operator
\[
B_h : \bV_{h,0} \to Q_h, \quad B_h(\bv_h)=\div \bv_h,
\]
is surjective, and so, there exists $\bw_h \in \bV_{h,0}$ such that

$$ \div \bw_h = q_h, $$
and moreover 
$\|\bw_h\|_{\bH^1(\Omega)} \lesssim \|q_h\|_{L^2(\Omega)}$.

\medskip

Now, we define the correction $\bG_h\in \bV_h$ as
\[
\bG_h := \bE_h - \bw_h.
\]

Hence
\[
\div \bG_h = \div \bE_h - \div \bw_h
= q_h - q_h = 0, 
\]
and since $\bw_h \in V_{h,0}$, we have 
\[
\int_{\Gamma} \bG_h \cdot \bn  = \int_{\Gamma} \bE_h \cdot \bn  =\int_{\Gamma}  \bg_h \cdot \bn =0. 
\]

Finally, by the triangle inequality, we get
\[
\|\bG_h\|_{H^1(\Omega)}
\le \|\bE_h\|_{H^1(\Omega)} + \|\bw_h\|_{H^1(\Omega)}
\lesssim  \|\bg_h\|_{H^{1/2}(\Gamma)}.
\]
\qed
\end{proof}

We are now in a position to derive an a priori estimate for $\bu_h$, which will be used in the subsequent analysis.

\begin{lemma}
Let $ \bu_h$ be the solution of \eqref{NS_fem}. Then, assuming that\\ $C_p - \Rey\, C_a C_g \, \| \bg_h \|_{H^{1/2}(\Gamma)} >0$,
we get

\begin{equation} \label{cota-uh}
\| \bu_h\|_{H^1(\Omega)} \leq \frac{  \|\bf\|_{H^{-1}(\Omega)}   +  C_g (1 + C_p ) \| \bg_h \|_{H^{1/2}(\Gamma)}  } {C_p -   \Rey C_a C_g\| \bg_h \|_{H^{1/2}(\Gamma)}    }
\end{equation}
\end{lemma}
\begin{proof}
Let us defined $\bu_h^0 \in \bV_{h,0}$ as
$$ \bu_h^0 = \bu_h - \bG_h,$$
where $\bG_h$ is the extension of $\bg_h$ given in Lemma \ref{extension_gh}.
Then, $ \forall\bw_h\in\bV_{h,0}$, we have
    \begin{equation} \label{ecuacion-uh0}
	 \begin{array}{rl}
        \int_\O\nabla\bu_h^0:\nabla\bw_h &+ \int_\O\nabla\bG_h:\nabla\bw_h+  \Rey\, \int_\O (\bu_h^0\cdot\nabla)\bu^0_h\cdot\bw_h +  \Rey\, \int_\O (\bu_h^0\cdot\nabla)\bG_h\cdot\bw_h \\
        &+  \Rey\, \int_\O (\bG_h\cdot\nabla)\bu^0_h\cdot\bw_h +  \Rey\, \int_\O (\bG_h\cdot\nabla)\bG_h\cdot\bw_h       
        -\int_\O p_h\div\bw_h\\
        &= \int_\O\bf\bw_h  \\
	\end{array} 
\end{equation}

Now, we observe that, since $\int_{\Omega} \div \bu_h q_h =0 $ for all $q_h \in Q_h$ and $\div \bG_h =0$ we get
$$\int_{\Omega} \div \bu_h^0  q_h = 0, \quad \forall q_h \in Q_h,$$
and so taking $q_h = \div \bu_h^0$ we can conclude that $\div \bu_h^0 =0$. Therefore,
taking $\bw_h = \bu_h^0 $ in \eqref{ecuacion-uh0} we get $\int_\O (\bu_h^0\cdot\nabla)\bu^0_h\cdot\bu_h^0=0$ and  $\int_\O (\bG_h\cdot\nabla)\bu^0_h\cdot\bu_h^0=0$. Then 
\begin{equation} 
	 \begin{array}{rl}
        \| \nabla\bu_h^0 \|_{L^2(\Omega)}^2 &\leq   \| \nabla \bG_h \|_{L^2(\Omega)} \| \nabla\bu_h^0 \|_{L^2(\Omega)}   
        +  \Rey\, C_a\, \|  \bG_h \|_{H^1(\Omega)} \|\bu_h^0 \|_{H^1(\Omega)}^2 \\
        &+  \Rey\, C_a\, \|  \bG_h \|^2_{H^1(\Omega)} \|\bu_h^0 \|_{H^1(\Omega)}
+ \|\bf\|_{H^{-1}(\Omega)} \|\bu_h^0 \|_{H^1(\Omega)} \\
	\end{array} 
\end{equation}
Since $\bu_h^0 \in \bH^1_0(\Omega)$,  the Poincaré inequality implies that $C_p \|\bu_h^0\|^2_{H^1(\Omega)} \leq \|\nabla \bu_h^0\|^2_{L^2(\Omega)}$, which  
together with the a priori estimation for the extension $\bG_h$ 
 and the small-data assumption, $C_p - \Rey\, C_a C_g \, \| \bg_h \|_{H^{1/2}(\Gamma)} >0$, allow us to write

$$  \|\bu_h^0\|_{H^1(\Omega)} \leq \frac{ C_g \| \bg_h \|_{H^{1/2}(\Gamma)}  +  \Rey\, C_a C_g^2 \| \bg_h \|^2_{H^{1/2}(\Gamma)}  + \|\bf\|_{H^{-1}(\Omega)} } {C_p -   \Rey\, C_a C_g\| \bg_h \|_{H^{1/2}(\Gamma)}    }, $$
and therefore

$$  \|\bu_h\|_{H^1(\Omega)} \leq \frac{  \|\bf\|_{H^{-1}(\Omega)}   +  C_g (1 + C_p ) \| \bg_h \|_{H^{1/2}(\Gamma)}  } {C_p -   \Rey\, C_a C_g\| \bg_h \|_{H^{1/2}(\Gamma)}    }. $$
\qed
\end{proof}

\begin{remark}
    We observe that if we consider problem \eqref{NSweak_regular_data} with $\tilde{\bg}=\bg_h$, and the extension $\bG_h$ of $\bg_h$  introduced in Lemma \ref{extension_gh} instead of $\bpsi$,  then we obtain an estimate for the solution $\bv$ analogous to \eqref{cota-uh}.
\end{remark}

\section{A posteriori error estimator}\label{sec:aposteriori}

In this section, we present an a posteriori estimator for the finite element solution of the regularized problem. From now on, we consider the case where $\O$ is convex. The goal is to prove both the efficiency and the reliability of the estimator by measuring the error in the $L^4$ norm. 

\begin{assumption}\label{assumption1}
    Let $D\subset\matR^2$ be a convex polygonal domain. Given $\bF\in \bL^{4/3}(D)$, the Stokes problem
    \begin{equation*}       
      \label{Stokes}
    	 \left\{\begin{array}{rcc}
    		-\Delta\phi+\nabla \pi  &= \bF & \textup{in } \, D \\
    		\div\phi  &=0 & \textup{in } \,  D \\
                \phi &= 0 & \textup{on } \, \partial D
    	\end{array}\right.
    \end{equation*}
    has a unique solution $(\phi,\pi)\in \bW^{2,4/3}(D) \times W^{1,4/3}(D)\setminus\matR$ (see \cite[Theorem I.5.4, Remark I.5.6]{girault1986} and \cite{grisvard1979}). Moreover, exists $C_A>0$ independent of $\bF$ such that
    $$\| \phi \|_{W^{2,4/3}(D)} + \| \pi \|_ {W^{1,4/3}(D) }\leq C_A \| \bF \|_{L^{4/3}(D)}$$
\end{assumption}

\begin{remark}\label{inclusion}
    By 1.4.4.5 in \cite{Grisvard1985}, we have $W^{k,4/3}(\O)\subset W^{k-1,4}(\O)$ for $k\geq 1$, and since $L^4(\O)\subset L^2(\O)$, we can conclude in particular that if $\O$ convex and $\phi \in \bW^{2,4/3}(\O)$, then $\phi\in \bH^1(\O)$. On the other hand, since $W^{1,4/3}(\O)\subset L^4(\O)$, we can conclude that if $\pi\in W^{1,4/3}(\O)$ and has zero mean, then $\pi\in L^2_0(\O)$. 
\end{remark}

For our a posteriori analysis, we make the following regularity assumption on the solution of problem \eqref{NSweak_regular_data}, which in particular implies the continuity of $\bv$.

\begin{assumption}\label{assumption2}
Let $\Omega \subset \mathbb{R}^2$ be a convex polygonal domain, let $\mathbf{f}\in L^2(\Omega)$, and let $\tilde{\bg}\in \bH^{1/2}(\Gamma)$ satisfy
$
\int_{\Gamma}\tilde{\bg}\cdot\bn = 0
$. We assume that the solution $(\bv,p)$ of problem \eqref{NSweak_regular_data}, with $\bv|_\Gamma=\tilde{\bg}$, belongs to
$
\bH^{1+s}(\Omega)\times \bigl(H^s(\Omega)\cap L^2_0(\Omega)\bigr)
$
for some $s>0$.
\end{assumption}

\begin{lemma}
    Let $\bz\in \bL^4(\O)$. Then $\bF=|\bz|^2\bz\in \bL^{4/3}(\O)$. Moreover, $\|\bF\|_{L^{4/3}(\O)}=\|\bz\|^3_{L^4(\O)}$.
\end{lemma}
\begin{proof} 
Clearly,  $\int_\O ||\bz|^2\bz|^{4/3} =  \int_\O|\bz|^4$  and hence
$$\|\bF\|_{L^{4/3}(\O)}=\left(\int_\O |\bF|^{4/3}\right)^{3/4}=\left(\int_\O |\bz|^4\right)^{3/4}=\|\bz\|^3_{L^4(\O)}.$$ 
\qed
\end{proof}

 We denote $I_h:W^{k,p}(\O)\to P_d(\O)\cap  H^1(\O)$  the Scott-Zhang interpolator \cite{SZ1990} with $d=1,2$.

For all $T\in\matT_h$, $1\leq p<\infty$, the following estimates hold (see \cite[inequality (4.3)]{SZ1990}):
\begin{equation}\label{estimate_interp_1}
    \|\varphi-I_h\varphi\|_{W^{m,p}(T)}\leq C_1 h_T^{k-m}\|\varphi\|_{W^{k,p}(\omega_T)} \quad \forall \varphi\in W^{k,p}(\Omega)
\end{equation}
with $0\leq m \leq k \leq d+1$.

We also have the following result.
\begin{prop}\label{estimacion_lado}
        Let $\varphi\in W^{2,p}(\O)$, $ 1 \leq p < \infty$.  For any $T\in\matT_h$ and any edge $\ell$ of $T$, it holds that
\begin{equation}\label{estimate_interp_2}
    \|\varphi-I_h\varphi\|_{L^p(\ell)}\leq C_2 |\ell| ^{2-\frac1p}\|\varphi\|_{W^{2,p}(\omega_T)} 
\end{equation}
\end{prop}

\begin{proof}
For any element $T \in \mathcal{T}_h$, any edge $\ell$ of $T$, and any function $v \in W^{1,p}(T)$, we have the following trace estimate (see, for example,  \cite[Lemma 2.1]{AAD}):
    $$\|v\|_{L^p(\ell)}\leq \left(\frac{2|\ell|}{|T|}\right)^{1/p}(\|v\|_{L^p(T)}+h_T\|\nabla v\|_{L^p(T)})$$

Since $\varphi-I_h\varphi\in W^{2,p}(T)$ we can use the previous result and get:
    $$\|\varphi-I_h\varphi\|_{L^p(\ell)}\leq \left(\frac{2|\ell|}{|T|}\right)^{1/p}(\|\varphi-I_h\varphi\|_{L^p(T)}+h_T\|\nabla (\varphi-I_h\varphi))\|_{L^p(T)}).$$ Using \eqref{estimate_interp_1}, we obtain
    $$\|\varphi-I_h\varphi\|_{L^p(\ell)}\leq C \left(\frac{2|\ell|}{|T|}\right)^{1/p}h_T^2\|\varphi\|_{W^{2,p}(\omega_T)}$$
    By the shape-regularity of the mesh, $h_T$ is comparable to $|\ell|$ and $h_T^2$ is comparable to $ |T|$ and so
 
    $$\|\varphi-I_h\varphi\|_{L^p(\ell)}\leq C |\ell|^{1/p}|\ell|^{2-2/p}\|\varphi\|_{W^{2,p}(\omega_T)}=C|\ell|^{2-1/p}\|\varphi\|_{W^{2,p}(\omega_T)}.$$
    \qed
\end{proof}

From now on, we assume that $\bf \in L^4 (\Omega)$ and we denote by $\bf_h$  an approximation of $\bf$ such that $\| \bf -\bf_h\|_{L^4(\O)}\to 0$ as $h\to 0$.
Next, for each $T\in\matT_h$, we define the local a posteriori error estimator as
\begin{equation}\label{estim_local_convexo}
        \eta_T := \left\{h_T^8\|  \mathcal{R}_T \|_{L^4(T)}^4+\sum_{\ell\in\matE_T}|\ell|^5 \|\bJ_\ell \|^4_{L^4(\ell)} +h_T^4\|\div\bu_h\|^4_{L^4(T)} \right\}^{1/4}
\end{equation}

where $$ \mathcal{R}_T=
\bf_h +\Delta\bu_h-\Rey\,(\bu_h\cdot\nabla)\bu_h-\nabla p_h,$$  and $$\bJ_\ell= \frac12 \left[\frac{\partial\bu_h} {\partial\bn}\right]_\ell $$ with
 $$\left[\frac{\partial\bu_h}{\partial\bn}\right]_\ell=\nabla(\bu_h|_{T_1})\cdot\bn -\nabla(\bu_h|_{T_2})\cdot\bn .$$ 

Here, if $\ell=T_1\cap T_2$ and $\bn_1,\bn_2$ denote the outward unit normal vectors to $T_1$ and $T_2$, respectively, then $ \bn_2 =- \bn_1$. In this case, we denote $\bn:=\bn_1$. 

 And the global error estimator $\eta_{\Omega}$ is defined as   
$$ \eta_{\Omega} = \left( \sum_{T\in\matT_h } \eta_T^4 \right) ^{1/4}$$

\begin{prop}\label{Robustness}
    (Robustness)
    Let $\bv$ and $\bu_h$ be solutions of \eqref{NSweak_regular_data} and \eqref{NS_fem}, respectively.  
    Under the assumption that $\Rey$ is small enough, the following estimate holds: 
    $$\|\bv-\bu_h\|_{L^4(\O)}\lesssim
    \eta_{\Omega}
    +h^2\|\bf-\bf_h\|_{L^4(\O)}.$$ 
\end{prop}

\begin{proof}
    Taking Assumption 1 into account, consider $\bF= |\bv-\bu_h|^2(\bv-\bu_h)$. Since $\bH^1(\O)\hookrightarrow \bL^4(\O)$ \cite{Adams1975,Grisvard1985}, we conclude that $\bv-\bu_h\in\bL^4(\O)$ and, moreover, by a previous lemma, we have $\bF\in \bL^{4/3}(\O)$ and $\|\bF\|_{L^{4/3}(\O)}=\|\bv-\bu_h\|^3_{L^4(\O)}$. 

    First, note that subtracting \eqref{NSweak_regular_data} and \eqref{NS_fem}, for every $\bv_h\in \bV_{h,0}$
  and every $q_h\in Q_h$, it follows that
    \begin{equation}\label{ec_error} 
    \left\{\begin{array}{r} \int_\O\nabla(\bv-\bu_h):\nabla\bv_h-\int_\O\div\bv_h(q-p_h)
    +\Rey\,\int_\O(\bv\cdot\nabla)\bv\cdot\bv_h\\ -
    \Rey\,\int_\O(\bu_h\cdot\nabla)\bu_h\cdot\bv_h=0 \\
    \\
    \int_\O \div(\bv-\bu_h) q_h =0 \end{array}\right. 
    \end{equation}
   
    Let $\phi$ and $\pi$ be as in Assumption 1. Since $\bv-\bu_h \in \bH^1_0(\O)$, integrating by parts and using that $\div \phi =0$, we obtain
    \begin{align*}
        \|\bv-\bu_h\|^4_{L^4(\O)} &= \int_\O |\bv-\bu_h|^2(\bv-\bu_h)\cdot (\bv-\bu_h) = \int_\O (\bv-\bu_h)\cdot(-\Delta\phi + \nabla\pi)\\
        &=\int_\O\nabla(\bv-\bu_h):\nabla \phi -\int_\O \div(\bv-\bu_h)\pi-\int_\O\div\phi(q-p_h)  = I
    \end{align*}
    Let $\mathbf{I}_h: \bW^{2,\frac43}(\O)\to \bV_{h,0}$ and $\mathcal{I}_h:\bW^{1,\frac43}(\O)\to Q_h$ be interpolators such that $\mathbf{I}_h (\bv) = (I_h v_1, I_h v_2)$, with $I_h$ and $\mathcal{I}_h$  satisfying \eqref{estimate_interp_1}. Using equation \eqref{ec_error} with $\bv_h = \mathbf{I}_h \phi $ and $q_h= \mathcal{I}_h \pi $, subtracting the remaining terms and integrating by parts on each triangle, we obtain
    \begin{align*}    
        I &=\sum_{T\in\matT_h}\left[\int_T \left(\Delta \bu_h -\nabla p_h-\Rey\,(\bu_h\cdot\nabla)\bu_h \right)\cdot(\phi-\bI_h\phi)\right.\\
        &\left.+ \int_T\div\bu_h (\pi-\mathcal{I}_h\pi) +\int_{\partial T} \left(\frac{\partial\bu_h}{\partial\bn}\right)\cdot(\phi-\bI_h\phi)\right]\\
        &-\Rey\,\int_\O (\bv\cdot\nabla)\bv\cdot\phi+\Rey\,\int_\O (\bu_h\cdot\nabla)\bu_h\cdot\phi \\
        &+ \int_\O \nabla \bv : \nabla (\phi-\bI_h\phi) + \Rey\,\int_{\Omega}(\bv\cdot\nabla)\bv (\phi-\bI_h\phi) - \int_{\O} q \div(\phi-\bI_h\phi)
    \end{align*}
        Here we used that $ \sum_{T\in\matT_h}\int_{\partial T} p_h \bn  (\phi - \mathbf{I}_h \phi) =0 $,
        because the normal components of $p_h$ are continuous and $\phi - \mathbf{I}_h \phi$ is continuous, since we are assuming that the functions in $Q_h$ and $\bV_h$ are continuous and $\phi \in \bW^{2,\frac43}(\O)$ is also continuous. We also used that $\int_{\O} \div \bv (\pi-\mathcal{I}_h\pi)=0$ because $\pi-\mathcal{I}_h\pi \in L^2_0(\O)$, and that $\bv-\bu_h$ and $\phi$ vanish on the boundary.
        We add and subtract $ \int_{\Omega} (\bv\cdot\nabla)\bu_h\phi$ and use equation \eqref{NSweak_regular_data}
        together with the approximation $\bf_h$ of $\bf$.
        \begin{align*}      
        I&=\sum_{T\in\matT_h}\left[\int_T (\bf_h +\Delta\bu_h-\Rey\,(\bu_h\cdot\nabla)\bu_h-\nabla p_h)\cdot (\phi-\bI_h\phi)+\int_T (\bf-\bf_h)(\phi-\bI_h\phi)\right.\\ 
       &\left.+ \int_T\div(\bu_h)(\pi-\mathcal{I}_h\pi) +\int_{\partial T} \left(\frac{\partial\bu_h}{\partial\bn}\right)\cdot(\phi-\bI_h\phi)\right]\\
       &+\Rey\,\int_\O (\bv\cdot\nabla)(\bu_h-\bv)\cdot\phi+\Rey\,\int_\O ((\bu_h-\bv)\cdot\nabla)\bu_h\cdot\phi\\
       &\leq \sum_{T\in\matT_h}\left(\|\bf_h +\Delta\bu_h-\Rey\,(\bu_h\cdot\nabla)\bu_h-\nabla p_h\|_{L^4(T)}\|\phi-\bI_h\phi\|_{L^{4/3}(T)}\right.\\
       &\left.+\|\div(\bu_h)\|_{L^4(T)}\|\pi-\mathcal{I}_h\pi\|_{L^{4/3}(T)}+\|\bf-\bf_h\|_{L^4(T)}\|\phi-\bI_h\phi\|_{L^{4/3}(T)}\right.\\
        & \left. + \frac12 \sum_{\ell\in\matE_T}\left\|\left[\frac{\partial\bu_h}{\partial\bn}\right]_\ell\right\|_{L^4(\ell)}\|\phi-\bI_h\phi\|_{L^{4/3}(\ell)} \right)+\Rey\,\|\bv\|_{L^4(\O)}\|\nabla\phi\|_{L^2(\O)}\|\bv-\bu_h\|_{L^4(\O)}\\
        &+\Rey\,\|\bv-\bu_h\|_{L^4(\O)}\|\nabla\bu_h\|_{L^2(\O)}\|\phi\|_{L^4(\O)}
    \end{align*}
    
    Using \eqref{estimate_interp_1} and \eqref{estimate_interp_2}, taking into account \eqref{inmersion} and Remark \ref{inclusion}, we obtain
    \begin{align*}
        I &\lesssim \sum_{T\in\matT_h}\left(h^2_T\|\mathcal{R}_T\|_{L^4(T)}\|\phi\|_{W^{2,4/3}(\omega_T)}+h_T\|\div(\bu_h)\|_{L^4(T)} \|\pi\|_{W^{1,4/3}(\omega_T) }\right.\\
        &\left. + \sum_{\ell\in\matE_T}|\ell|^{5/4}\left\|\bJ_\ell\right\|_{L^4(\ell)}\|\phi\|_{W^{2,4/3}(\omega_T)} \right)+\sum_{T\in\matT_h} h^2_T\|\bf-\bf_h\|_{L^4(T)}\|\phi\|_{W^{2,4/3}(\omega_T)}\\
        & +\Rey\,\|\bv\|_{L^4(\O)}\|\bv-\bu_h\|_{L^4(\O)}\|\phi\|_{H^1(\O)}+ \Rey\,\|\bu_h\|_{H^1(\O)}\|\bv-\bu_h\|_{L^4(\O)}\|\phi\|_{L^{4}(\O)}\\
        &  \lesssim \sum_{T\in\matT_h}\left(h^2_T\|\mathcal{R}_T\|_{L^4(T)}+h_T\|\div(\bu_h)\|_{L^4(T)}+ \sum_{\ell\in\matE_T}|\ell|^{5/4}\left\|\bJ_\ell\right\|_{L^4(\ell)} \right) \left( \|\phi\|_{W^{2,4/3}(\omega_T)} \right.\\ 
        &  \left. +\|\pi\|_{W^{1,4/3}(\omega_T)} \right) +\sum_{T\in\matT_h} h^2_T\|\bf-\bf_h\|_{L^4(T)}\|\phi\|_{W^{2,4/3}(\omega_T)}\\  &+\Rey\,\|\bv\|_{L^4(\O)}\|\bv-\bu_h\|_{L^4(\O)}\|\phi\|_{H^1(\O)} + \Rey\,\|\bu_h\|_{H^1(\O)}\|\bv-\bu_h\|_{L^4(\O)}\|\phi\|_{L^{4}(\O)}
    \end{align*}
Now, by Holder inequality  $\sum_{j=1}^m|a_jb_j|\leq\left(\sum_{j=1}^ma_j^p\right)^{1/p}\left(\sum_{j=1}^mb_j^q\right)^{1/q}$,  with $\frac1p+\frac1q=1$, we can write  

\begin{align*}
 I&\lesssim  \left\{ \sum_{T\in\matT_h} \left( h_T^2\|\mathcal{R}_T\|_{L^4(T)}+h_T\|\div\bu_h\|_{L^4(T)}+\sum_{\ell\in\matE_T}|\ell|^{5/4}\left\|\bJ_\ell\right\|_{L^4(\ell)}\right)^4 \right\}^{1/4}\\ 
&\left\{ \sum_{T\in\matT_h} \left( \|\phi\|_{W^{2,4/3}(\omega_T)} \right.  \left. +\|\pi\|_{W^{1,4/3}(\omega_T)}\right)^{4/3}\right\}^{3/4}+\sum_{T\in\matT_h} h^2_T\|\bf-\bf_h\|_{L^4(T)}\|\phi\|_{W^{2,4/3}(\omega_T)}\\
 & +\Rey\,\|\bv\|_{L^4(\O)}\|\bv-\bu_h\|_{L^4(\O)}\|\phi\|_{H^1(\O)}+ \Rey\,\|\bu_h\|_{H^1(\O)}\|\bv-\bu_h\|_{L^4(\O)}\|\phi\|_{L^{4}(\O)}\\
     &\lesssim  \left\{ \sum_{T\in\matT_h} \left( h_T^8\|\mathcal{R}_T\|^4_{L^4(T)}+h_T^4\|\div\bu_h\|^4_{L^4(T)}+\sum_{\ell\in\matE_T}|\ell|^{5}\left\|\bJ_\ell\right\|^4_{L^4(\ell)}\right) \right\}^{1/4}\\
&\left\{ \sum_{T\in\matT_h} \left( \|\phi\|^{4/3}_{W^{2,4/3}(\omega_T)} \right.\left. +\|\pi\|^{4/3}_{W^{1,4/3}(\omega_T)}+\|\phi\|^{4/3}_{W^{2,4/3}(\omega_T) }\right)\right\}^{3/4}\\
&+\sum_{T\in\matT_h} h^2_T\|\bf-\bf_h\|_{L^4(T)}\|\phi\|_{W^{2,4/3}(\omega_T)}\\
        & +\Rey\,\|\bv\|_{L^4(\O)}\|\bv-\bu_h\|_{L^4(\O)}\|\phi\|_{H^1(\O)}+ \Rey\,\|\bu_h\|_{H^1(\O)}\|\bv-\bu_h\|_{L^4(\O)}\|\phi\|_{L^{4}(\O)}
    \end{align*}
   
 Then, by using that 
the triangulation satisfies the minimum angle condition,  we have  
    
    $$\sum_{T\in\matT_h}\|\phi\|^{\frac43}_{W^{2,\frac43}(\omega_T)} \lesssim \|\phi\|^{\frac43}_{W^{2,\frac43}(\Omega)},
    \quad  \sum_{T\in\matT_h}\|\pi\|^{\frac43}_{W^{1,\frac43}(\omega_T)} \lesssim \|\pi\|^{\frac43}_{W^{2,\frac43}(\Omega)},$$
and by using this join with the a priori estimate
 $$\|\phi\|_{W^{2,4/3}(\O)}+\|\pi\|_{W^{1,4/3}(\O)}\lesssim \|\bF\|_{L^{4/3}(\O)}=\|\bv-\bu_h\|^3_{L^4(\O)}$$
 we get    
  \begin{align*}
        I&\lesssim  \left\{ \sum_{T\in\matT_h} \eta_T^4 \right\}^{1/4} \|\bv-\bu_h\|^3_{L^4(\O)}
        +  h^2 \|\bf-\bf_h\|_{L^4(\Omega)}
        \|\bv-\bu_h\|^3_{L^4(\O)}\\
        & +\Rey\,\|\bv\|_{L^4(\O)}\|\bv-\bu_h\|_{L^4(\O)}\|\phi\|_{H^1(\O)}+ \Rey\,\|\bu_h\|_{H^1(\O)}\|\bv-\bu_h\|_{L^4(\O)}\|\phi\|_{L^{4}(\O)}
    \end{align*}
    Since
    $$\|\phi\|_{H^1(\Omega)},\,\|\phi\|_{L^4(\Omega)} \lesssim \|\phi\|_{W^{2,4/3}(\Omega)},$$
    combining this with the a priori estimate
    $$\|\phi\|_{W^{2,4/3}(\Omega)} + \|\pi\|_{W^{1,4/3}(\Omega)} 
    \lesssim \|\mathbf{F}\|_{L^{4/3}(\Omega)}
    = \|\bv-\mathbf{u}_h\|_{L^4(\Omega)}^{3},$$
    we obtain
    \begin{align*}
        I&\lesssim  \left\{ \sum_{T\in\matT_h} \eta_T^4 \right\}^{1/4} \|\bv-\bu_h\|^3_{L^4(\O)}
        +  h^2 \|\bf-\bf_h\|_{L^4(\Omega)}
        \|\bv-\bu_h\|^3_{L^4(\O)}\\
        & +\Rey\,(\|\bv\|_{H^1(\O)}+\|\bu_h\|_{H^1(\O)})\|\bv-\bu_h\|^4_{L^4(\O)}
    \end{align*}
    
    From   \eqref{cota_v}  and   \eqref{cota-uh} 
    we know that there exist constants $M_v$ and $M_u$, depending on the data and the Reynolds number $\Rey$,  such that $\|\bv\|_{H^1(\O)} \leq M_v$ and $\|\bu_h\|_{H^1(\O)} \leq M_u$. Therefore, assuming that $\Rey$ is small enough, it follows that

    \begin{align*}
        \|\bv-\bu_h\|_{L^4(\O)} &\lesssim \left( \sum_{T\in\matT_h} \eta_T^4 \right)^{1/4}+h^2\|\bf-\bf_h\|_{L^4(\O)},
    \end{align*}
   \qed
\end{proof}

Now the goal is to prove the efficiency of the estimator.

For $T \in \matT_h$, we define the bubble function $b_T$ by
\[
b_T :=
\begin{cases}
27\lambda_1^T \lambda_2^T \lambda_3^T  & \text{in } T, \\[6pt]
0  & \text{in } \Omega \setminus T,
\end{cases}
\]
where $\lambda_1^T$, $\lambda_2^T$ and $\lambda_3^T$ denote the barycentric coordinates of $T$.

To demonstrate properties of bubble functions, we will use the fact that for any $T\in\matT_h$ 
 $$\int_T \lambda_{1,T}^{n_1}\lambda_{2,T}^{n_2} \lambda_{3,T}^{n_3} dx = \frac{n_1! n_2! n_3!2!}{(n_1+n_2+n_3+ 2)!}|T|.$$

\begin{lemma}\label{propiedades_burbuja}
Let $T\in\matT_h$. Then $b_T$ satisfies:
\begin{itemize}
   \item[(a)] $0\leq b_T\leq 1$. 
   \item[(b)] $\|b_T\|_{L^p(T)}\lesssim |T|^{1/p}, \quad for 1\leq p < \infty$
  \item[(c)] For any $w \in P_k(T) $ and $\alpha >0$, 
  $
\|w\|_{L^p(T)}\lesssim \|wb_T^{\frac{\alpha}{p}}\|_{L^p(T)}.
$

\end{itemize}
\end{lemma}

\begin{proof}
Property (a) follows directly from the definition. For $T\in\mathcal{T}_h$ we denote by $\hat{T}$ a reference triangle and by $B$ the matrix associated with the affine mapping $F: \hat{T} \rightarrow T$, $F (\hat{T}) = T$, the bubble function on the reference triangle is denoted by
$\hat{b}_{\hat{T}} $.

\begin{align*}
    \|b_T\|_{L^p(T)}^p 
    &= \int_{T} b_T^p 
    = \int_{\hat{T}} \hat{b}_{\hat{T}}^p |\det B| = 2|T| \int_{\hat{T}} \hat{b}_{\hat{T}}^p,
\end{align*}
and (b) holds.

Now, we will prove (c). For any $w\in P_k(T)$, we take $\hat{w}\in P_k(\hat{T})$ as $\hat{w}= w \circ F$. Let as define 
$|||\hat{w}|||_{L^p(\hat{T})}=\|\hat{w}b_{\hat{T}}^{\frac{\alpha}{p}}\|_{L^p(\hat{T})}$, with $\alpha>0$. Since in finite dimension all norms are equivalent, there exist constants $\hat{c}_1$ and $\hat{c}_2$ (depending on $\hat{T}$ and $k$) such that
$$
\hat{c}_1\|\hat{w}\|_{L^p(\hat{T})}\leq|||\hat{w}|||_{L^p(\hat{T})}\leq\hat{c}_2\|\hat{w}\|_{L^p(\hat{T})}.
$$
Taking this into account and performing a change of variables we obtain
$$
\|w\|_{L^p(T)}=|\det B|^{1/p}\|\hat{w}\|_{L^p(\hat{T})}
\leq |\det B|^{1/p} \frac{1}{\hat{c}_1}||| \hat{w}|||_{L^p(\hat{T})}  =
|\det B|^{1/p} \frac{1}{\hat{c}_1}\|\hat{w}b_{\hat{T}}^{\frac{\alpha}{p}}\|_{L^p(\hat{T})} ,
$$
and changing variables again we get
$$
\|w\|_{L^p(T)}\lesssim \|wb_T^{\frac{\alpha}{p}}\|_{L^p(T)}.
$$
\qed
\end{proof}

For our analysis, we recall the following inverse estimate (see \cite[Lemma 4.5.3]{BS1994}). Let $P$ be a finite-dimensional subspace of $W^{l,p}(T)\cap W^{m,q}(T)$, with $1\leq p,q<\infty$, $T\in\mathcal{T}_h$, and $0\leq m \leq l$. Then there exists a positive constant $C$, independent of $h_T$, such that, for all $v\in P$, 
\begin{equation}\label{inverse_ineq}
\|v\|_{W^{l,p}(T)} \leq C\, h_T^{m-l+2(1/p-1/q)} \|v\|_{W^{m,q}(T)}.
\end{equation}

\begin{lemma}\label{Cota_div}
 Let $\bv$ and $\bu_h$ be solutions of \eqref{NSweak_regular_data} and \eqref{NS_fem}, respectively. Then, 
\[
h_T\|\div\bu_h\|_{L^4(T)}\lesssim \|\bv-\bu_h\|_{L^4(T)}.
\]
\end{lemma}

\begin{proof}
Let $q_T=(\div \bu_h)^3b_T$. 

Using  item (c) of Lemma \ref{propiedades_burbuja} and the fact that $\div \bv = 0$ (since by the compatibility condition and the divergence theorem one has $\div\bv\in L^2_0(\O)$ and using it as a test function yields the result), we obtain
\begin{align*}
\|\div\bu_h\|^4_{L^4(T)}
&\lesssim\|\div\bu_h b_T^{1/4}\|^4_{L^4(T)}
= \int_T(\div\bu_h)^4b_T \\
&=\int_T\div\bu_h q_T
=\int_T\div(\bu_h-\bv) q_T = \int_T (\bu_h-\bv)\nabla q_T.
\end{align*}
Thus,

$$\|\div\bu_h\|^4_{L^4(T)} \lesssim \|\bu_h-\bv\|_{L^4(T)}\|\nabla q_T\|_{L^{4/3}(T)}\\
\lesssim  \|\bu_h-\bv\|_{L^4(T)}h^{-1}_T\|q_T\|_{L^{4/3}(T)},
$$
where in the last inequality we used \eqref{inverse_ineq}. Then, 
\begin{align*}
\|q_T\|_{L^{4/3}(T)}
&=\|(\div\bu_h)^3b_T\|_{L^{4/3}(T)}
= \left(\int_T(\div\bu_h )^4b_T^{4/3}\right)^{3/4}\\
&\leq \left(\int_T(\div\bu_h)^4\right)^{3/4}
=\|\div\bu_h\|^3_{L^4(T)},
\end{align*}

Therefore
\[
\|\div\bu_h\|^4_{L^4(T)}
\lesssim \|\bu_h-\bv\|_{L^4(T)}h^{-1}_T\|\div\bu_h\|^3_{L^4(T)},
\]
which implies
\[
h_T\|\div\bu_h\|_{L^{4}(T)}\lesssim\|\bu_h-\bv\|_{L^4(T)}.
\]
\qed
\end{proof}

\begin{lemma}\label{Cota de R_T}
Let $\bv$ and $\bu_h$ be solutions of \eqref{NSweak_regular_data} and \eqref{NS_fem}, respectively.   Under the assumption that $\Rey$ is small enough, we have 
\[ {
h_T^2\|R_T\|_{L^4(T)}
}\lesssim 
\|\bv-\bu_h\|_{L^4(T)}
+
\|q-p_h\|_{W^{-1,4}(T)}
+
h_T^2\|\bf-\bf_h\|_{L^4(T)}
.
\]

\end{lemma}

\begin{proof}
Let

$$\bw_T= R_T^3 b^2_T =   (\bf_h+\Delta\bu_h-\Rey\,(\bu_h\cdot\nabla)\bu_h-\nabla p_h)^3b^2_T .$$ 

note that, 

$$\bw_T= 0 \quad \text{on} \quad \Gamma, \quad \text{and} \quad \nabla \bw_T = 0 \quad \text{on} \quad \Gamma$$

Using Lemma \ref{propiedades_burbuja} and the fact that $\bw_T \in \bH^1_0(T) \subset \bV_{h,0}$ we obtain

\begin{align*}
\| R_T\|^4_{L^4(T)}
&\lesssim
\|R_T b_T^{2/4}\|^4_{L^4(T)} =
\int_T(\bf_h+\Delta\bu_h-\Rey\,(\bu_h\cdot\nabla)\bu_h-\nabla p_h)\bw_T . \\
&=\int_T \Big(\Delta\bu_h-\Rey\,(\bu_h\cdot\nabla)\bu_h + \Rey\,(\bv\cdot\nabla)\bv
-\nabla p_h\Big)\bw_T  + \int_T \nabla \bv : \nabla \bw_T\\
& - \int_T q \, \div \bw_T +\int_T(\bf_h - \bf)\bw_T.
\end{align*}

Applying integration by parts, $\nabla {\bw_T}|_{\Gamma} = 0$, Hölder's inequality and \eqref{inverse_ineq}, we obtain

\begin{align*}
\int_T\Delta\bu_h\bw_T +  \int_T \nabla \bv : \nabla \bw_T
&=
\int_T(\bu_h-\bv)\Delta\bw_T
\leq
\|\bu_h-\bv\|_{L^4(T)}
\|\Delta\bw_T\|_{L^{4/3}(T)} \\
&\lesssim\|\bu_h-\bv\|_{L^4(T)}h_T^{-2}\|\bw_T\|_{L^{4/3}(T)} .
\end{align*}

Moreover, recalling that for $q\in W^{-1,4}(\Omega)$
 
 $$\|q\|_{W^{-1,4}(\Omega)}=\sup_{\psi\in W_0^{1,\frac43}(\Omega)\setminus\{0\}}
 \frac{\int_\Omega q\psi}{\|\psi\|_{W^{1,\frac43}(\Omega)}},$$

integrating by parts and using \eqref{inverse_ineq} again, it follows that

\begin{align*}
\int_T q \, \div \bw_T  + \int_T\nabla p_h \bw_T 
&=
\int_T(q - p_h)\div\bw_T \leq
\|q-p_h\|_{W^{-1,4}(T)}
\|\div\bw_T\|_{W^{1,4/3}(T)} \\
&\lesssim\|q-p_h\|_{W^{-1,4}(T)}h_T^{-2}\|\bw_T\|_{L^{4/3}(T)} .
\end{align*}

Finally, from  Lemma \ref{prop_oprerador_no_lineal} and \eqref{inverse_ineq} we get

\begin{align*}
\Rey\int_T\big((\bv\cdot\nabla)\bv-(\bu_h\cdot\nabla)\bu_h\big)\bw_T
&=
\Rey\int_T
\big((\bv\cdot\nabla)\bv-(\bv\cdot\nabla)\bu_h
+
(\bv\cdot\nabla)\bu_h-(\bu_h\cdot\nabla)\bu_h\big)\bw_T \\
&=
\Rey\int_T  (\bv-\bu_h)\cdot\nabla)\bu_h\cdot\bw_T
+
\Rey\int_T (\bv\cdot\nabla)(\bv-\bu_h)\cdot\bw_T \\
&=
\Rey\int_T ((\bv-\bu_h)\cdot\nabla)\bu_h\cdot\bw_T 
-
\Rey\int_T(\bv\cdot\nabla)\bw_T\cdot(\bv-\bu_h) \\
&\leq
\Rey\,\|\bv-\bu_h\|_{L^4(T)}
\|\nabla\bu_h\|_{L^2(T)}
\|\bw_T\|_{L^4(T)} \\
&+
\Rey\,\|\bv\|_{L^4(T)}
\|\nabla\bw_T\|_{L^2(T)}
\|\bv-\bu_h\|_{L^4(T)} \\
&\lesssim
\Rey\,\|\bv-\bu_h\|_{L^4(T)}
\|\nabla\bu_h\|_{L^2(T)}
h_T^{-1}\|\bw_T\|_{L^{4/3}(T)} \\
&+
\Rey\,\|\bv\|_{H^1(T)}
h_T^{-3/2}
\|\bw_T\|_{L^{4/3}(T)}
\|\bv-\bu_h\|_{L^4(T)} .
\end{align*}

Since
$$\|\bw_T\|_{L^{4/3}(T)}\leq \|\bf_h+\Delta\bu_h-\Rey\,(\bu_h\cdot\nabla)\bu_h-\nabla p_h\|^3_{L^4(T)},$$
we obtain
\begin{align*}
h_T^2\|\bf_h+\Delta\bu_h-\Rey\,(\bu_h\cdot\nabla)\bu_h-\nabla p_h\|_{L^4(T)}
&\lesssim
\|\bv-\bu_h\|_{L^4(T)}
+
\|q-p_h\|_{W^{-1,4}(T)} \\
&+
h_T\,\Rey\,\|\nabla\bu_h\|_{L^2(T)}
\|\bv-\bu_h\|_{L^4(T)} \\
&+
h_T^{1/2}\,\Rey\,
\|\bv\|_{H^1(T)}
\|\bv-\bu_h\|_{L^4(T)} \\
&+
h_T^2\|\bf-\bf_h\|_{L^4(T)}.
\end{align*}
Then, using estimates \eqref{cota_v} and \eqref{cota-uh}, together with the assumption that
$\Rey$ is sufficiently small, we conclude the proof.
\qed
\end{proof}

For $\ell \in \mathcal{E}_h$, we denote by $T_1$ and $T_2$ the two triangles sharing $\ell$, and we enumerate the vertices of $T_1$ and $T_2$ so that the vertices of $\ell$ are numbered first. We then consider the piecewise bubble function associated with the edge $\ell$, defined by
\[
b_\ell :=
\begin{cases}
4\lambda_1^{T_i}\lambda_2^{T_i}, & \text{in } T_i, \quad i=1,2, \\[6pt]
0, & \text{in } \Omega \setminus \omega_\ell,
\end{cases}
\]
with $\omega_\ell = T_1 \cup T_2$.
This bubble function satisfies
\begin{align*}
\int_\ell b_\ell &= C|\ell|,\\
\int_T b_\ell &= C|T|, \qquad T\in\omega_\ell,\\
\|b_\ell\|_{L^p(T)} &\leq C|T|^{1/p}, \qquad .
\end{align*}

\begin{lemma}\label{lema_burbuja_lado}
Let $w$ be a polynomial of degree $k$. Then there exists a constant $C$, depending only on  $ \sigma $, $k$ and $p$, such that
 for any $m \geq 1$ we get
\begin{itemize}
\item[i)] 
$
\| w\|_{L^p{(\ell)}}
\leq
C \|w b^m_{\ell}\|_{L^p(\ell)} .
$
\item[ii)] $
\|b^m_\ell w\|_{L^p{(T)}}
\leq
C |\ell|^{1/p}\|w\|_{L^p(\ell)} .
$
\end{itemize}
\end{lemma}

\begin{proof}
The proof follows the same argument as in Lemma \ref{propiedades_burbuja} and in \cite[Proposition 1.4 and Proposition 3.37]{verfurth2013}, based on a rescaling argument and the equivalence of norms on the reference elements.
\qed
\end{proof}

\begin{lemma}\label{Cota_salto}
Let $\bv$ and $\bu_h$ be solutions of \eqref{NSweak_regular_data} and \eqref{NS_fem}, respectively.  Assuming that Assumption \ref{assumption2} holds and 
$\Rey$ are sufficiently small, we have 

\[
|\ell|^{5/4}\|\bJ_\ell\|_{L^4(\ell)}
\lesssim
\|\bv-\bu_h\|_{L^4(\omega_\ell)}
+
\|q-p_h\|_{W^{-1,4}(\omega_\ell)}
+
|\ell|^2\|\bf-\bf_h\|_{L^4(\omega_\ell)}  
\]

\end{lemma}

\begin{proof}
Let $\Phi_\ell \in H^1_0(\omega_{\ell}) \subset \bV_{h,0}$ be defined as $\Phi_\ell = \bJ_\ell^3 b_\ell^2$. From Lemma \ref{lema_burbuja_lado} we have that

\begin{align*}
\|\bJ_\ell\|^4_{L^4(\ell)}
&\lesssim
\int_\ell 2 \bJ_\ell \Phi_\ell
=
\int_\ell
\left[\frac{\partial \bu_h}{\partial \bn}\right]_\ell
\Phi_\ell
= \int_{\omega_\ell}\nabla\bu_h : \nabla\Phi_\ell
+
\sum_{T\subset\omega_\ell}\int_T\Delta\bu_h \Phi_\ell
\\
&=
\int_{\omega_\ell}\nabla\bu_h : \nabla\Phi_\ell
-
\int_{\omega_\ell} p_h \div\Phi_\ell\\
&+
\sum_{T\subset\omega_\ell}
\int_T
\left[
(\Delta\bu_h-\Rey(\bu_h\cdot\nabla)\bu_h-\nabla p_h)\Phi_\ell
+
\Rey(\bu_h\cdot\nabla)\bu_h\Phi_\ell
\right]
\\
&=
\int_{\omega_\ell}\nabla\bu_h : \nabla\Phi_\ell
-
\int_{\omega_\ell} p_h\div\Phi_\ell
+
\Rey\int_{\omega_\ell} (\bu_h\cdot\nabla)\bu_h\cdot\Phi_\ell
\\
&\quad
+
\sum_{T\subset\omega_\ell}
\int_T
(\Delta\bu_h-\Rey(\bu_h\cdot\nabla)\bu_h-\nabla p_h)\Phi_\ell
+
\int_{\omega_\ell}\bf\cdot\Phi_\ell
-
\int_{\omega_\ell}\bf\cdot\Phi_\ell
\\
&=
\int_{\omega_\ell}(\nabla\bu_h-\nabla\bv):\nabla\Phi_\ell
-
\int_{\omega_\ell}(p_h-q)\div\Phi_\ell
\\
&\quad
+
\Rey\,\int_{\omega_\ell}(\bu_h\cdot\nabla)\bu_h\cdot\Phi_\ell
-
\Rey\,\int_{\omega_\ell}(\bv\cdot\nabla)\bv\cdot\Phi_\ell
\\
&\quad
+
\sum_{T\subset\omega_\ell}
\int_T
(\bf+\Delta\bu_h-\Rey(\bu_h\cdot\nabla)\bu_h-\nabla p_h)\Phi_\ell .
\end{align*}

Integrating by parts in each triangle $T\subset \omega_{\ell}$  we have
\begin{align*}
\|\bJ_\ell\|^4_{L^4(\ell)}
\lesssim \sum_{T\subset\omega_\ell}  
 &  \left\{ -\int_T(\bu_h-\bv)\Delta\Phi_\ell +
\int_{\partial T} (\bu_h-\bv)\frac{\partial\Phi_\ell}{\partial\bn}
- \int_{T}(p_h-q)\div\Phi_\ell  \right.
\\   &  \left. +
\Rey\int_{T}(\bv\cdot\nabla)(\bv-\bu_h)\cdot\Phi_\ell +
\Rey\int_{T}((\bv-\bu_h)\cdot\nabla)\bu_h\cdot\Phi_\ell  \right.
\\   & \left.
+ \int_T (\bf+\Delta\bu_h-\Rey(\bu_h\cdot\nabla)\bu_h-\nabla p_h)\Phi_\ell  \right\}
\end{align*}

We observe that, $ \sum_{T\subset\omega_\ell} 
\int_{\partial T} (\bu_h-\bv)\frac{\partial\Phi_\ell}{\partial\bn} =0$  because $\frac{\partial\Phi_\ell}{\partial\bn}$ is continuous,  
$\bv - \bu_h$ is continuous since Assumption \ref{assumption2} holds, and  $ \frac{\partial\Phi_\ell}{\partial\bn} =0$ in the rest of the edges. Adding and subtracting $\bf_h$ and applying the estimates given in Lemma \ref{prop_oprerador_no_lineal}, we obtain

\begin{align*}
\|\bJ_\ell\|^4_{L^4(\ell)}
\lesssim  &  \sum_{T\subset\omega_\ell}   \left\{
\|\bv-\bu_h\|_{L^4(T)}
    \|\Delta\Phi_\ell\|_{L^{4/3}(T)} + 
\|q-p_h\|_{W^{-1,4}(T)}
\|\div\Phi_\ell\|_{W^{1,4/3}(T)}  \right. \\
&  \left. +
\|\bf_h+\Delta\bu_h-\Rey(\bu_h\cdot\nabla)\bu_h-\nabla p_h\|_{L^4(T)}
\|\Phi_\ell\|_{L^{4/3}(T)} \right.
\\
&  \left. + \Rey\|\bv-\bu_h\|_{L^4(T)}
\|\nabla\bu_h\|_{L^2(T)}
\|\Phi_\ell\|_{L^4(T)}  \right.
\\
&  \left. +
\Rey\|\bv-\bu_h\|_{L^4(T)}
\|\Phi_\ell\|_{L^2(T)}
\|\bv\|_{L^4(T)}
+
\|\bf-\bf_h\|_{L^4(T)}
\|\Phi_\ell\|_{L^{4/3}(T)} \right\} .
\end{align*}

Now, using inverse inequalities \eqref{inverse_ineq} together with Lemma \ref{propiedades_burbuja} and the fact that $|\ell| \sim h_T$, for $T\in \omega_{\ell}$, we have
\begin{align*}
\|\Delta\Phi_\ell\|_{L^{4/3}(T)}
&\lesssim  h_T^{-2}\|\Phi_\ell\|_{L^{4/3}(T)}
\lesssim  |\ell|^{-2+\frac34}\|\bJ_\ell^3\|_{L^{4/3}(\ell)}
=
C\,|\ell|^{-\frac54}\|\bJ_\ell\|_{L^4(\ell)}^3,
\\
\|\div\Phi_\ell\|_{W^{1,4/3}(T)}
&\lesssim h_T^{-2}\|\Phi_\ell\|_{L^{4/3}(T)}
\lesssim |\ell|^{-\frac54}\|\bJ_\ell\|_{L^4(\ell)}^3,
\\
\|\Phi_\ell\|_{L^{4/3}(T)} &\lesssim \,|\ell|^{\frac34}\|\bJ_\ell\|^3_{L^4(\ell)}
\\
\|\Phi_\ell\|_{L^{4}(T)}
&\lesssim
h_T^{-1}\|\Phi_\ell\|_{L^{4/3}(T)}
\lesssim
|\ell|^{-\frac14}\|\bJ_\ell\|_{L^{4}(\ell)}^3,
\\
\|\Phi_\ell\|_{L^{2}(T)}
&\lesssim
h_T^{-\frac12}\|\Phi_\ell\|_{L^{4/3}(T)}
\lesssim
|\ell|^{\frac14}\|\bJ_\ell\|_{L^{4}(\ell)}^3 .
\end{align*}

Therefore,
\begin{align*}
|\ell|^{5/4}\|\bJ_\ell\|_{L^4(\ell)}
&\lesssim \sum_{T\subset\omega_\ell} \left\{
\|\bv-\bu_h\|_{L^4(T)} +
\|q-p_h\|_{W^{-1,4}(T)}  \right.
\\
& \left. + h_T^2
\|\bf_h+\Delta\bu_h-\Rey\,(\bu_h\cdot\nabla)\bu_h-\nabla p_h\|_{L^4(T)}  \right.
\\
& \left.  + |\ell|
\Rey\|\bv-\bu_h\|_{L^4(T)}
\|\nabla\bu_h\|_{L^2(T)}
+ |\ell|^{\frac23}
\Rey\|\bv-\bu_h\|_{L^4(T)}
\|\bv\|_{L^4(T)}  \right.
\\
& \left.
+ |\ell|^2
\|\bf-\bf_h\|_{L^4(T)}  \right\}
\end{align*}

Thus, using Lemma \ref{Cota de R_T} we conclude that

\begin{align*}
|\ell|^{5/4}\|\bJ_\ell\|_{L^4(\ell)}
&\lesssim 
\|\bv-\bu_h\|_{L^4(\omega_{\ell})}
+
\|q-p_h\|_{W^{-1,4}(\omega_{\ell})}\\
&+ |\ell|
\Rey\|\bv-\bu_h\|_{L^4(\omega_{\ell})}
\|\nabla\bu_h\|_{L^2(\omega_{\ell})}
\\
&+ |\ell|^{\frac23}
\Rey\|\bv-\bu_h\|_{L^4(\omega_{\ell})}
\|\bv\|_{L^4(\omega_{\ell})}
+ |\ell|^2
\|\bf-\bf_h\|_{L^4(\omega_{\ell})},
\end{align*}



and the proof concludes by using the apriori estimates \eqref{cota_v} and \eqref{cota-uh} together with the assumption that
the Reynolds number $\Rey$ is small enough.
\qed
\end{proof}

\begin{prop}\label{Efficiency}
(Efficiency)
Let $\bv$ and $\bu_h$ be solutions of \eqref{NSweak_regular_data} and \eqref{NS_fem}, respectively.  Assuming that Assumption \ref{assumption2} holds and  $\Rey$ is sufficiently small, we have 
\[
\eta_T
\lesssim
\|\bv-\bu_h\|_{L^4(\omega_T)}
+
\|q-p_h\|_{W^{-1,4}(\omega_T)}
+
h_T^2\|\bf-\bf_h\|_{L^4(\omega_T)} .
\]
\end{prop}

\begin{proof}
The result follows by combining the previous lemmas \ref{Cota_div}, \ref{Cota de R_T}  and \ref{Cota_salto}.
\qed
\end{proof}
 
Now, we are in condition to prove the main result.

\begin{theorem}
  Let $\bu$ be the very weak solution of \eqref{NS_dato_g} and let $\bu_h$ be the solution of \eqref{NS_fem}.  Assuming that Assumption 1 and Assumption 2 hold and  $\Rey$ is small enough, we have 

 $$ \| \bu - \bu_h \|_{L^4(\O)} \lesssim \eta_\O +h^2\|\bf-\bf_h\|_{L^4(\O)} + \| \bg - \bg_h\|_{L^2(\O)},
 $$
 and
 $$
\eta_\O \lesssim \|\bu-\bu_h\|_{L^4(\O)}
+ \|q-p_h\|_{W^{-1,4}(\O)} +
h^2\|\bf-\bf_h\|_{L^4(\O)}  +  \| \bg - \bg_h\|_{L^2(\O)}
$$
\end{theorem}
\begin{proof}
It follows from Propositions \ref{Robustness}, \ref{Efficiency} and the apriori estimate \eqref{error:u-v}.
\qed
\end{proof}

\section{Numerical example}\label{sec:example}

In this section, we illustrate the performance of the proposed adaptive finite element method and the associated a posteriori error estimator for a benchmark problem with low-regularity Dirichlet boundary data. In particular, we consider the classical \emph{lid-driven cavity flow problem}, where the discontinuous behavior of the boundary velocity at the upper corners provides a challenging test case for our adaptive algorithms.

Let $\O=[-1, 1]\times [-1, 1]$, $\Rey = 0.1$, $\bf=(0,0),$ and

\[
\mathbf{g}(x,y)=
\begin{cases}
(0,0), & -1<x<1,\quad y=-1,\\[2mm]
(1,0), & y=1,\\[2mm]
(0,0), & x=\pm1,\quad -1<y<1.
\end{cases}
\]    

For the numerical approximation, we employ Taylor-Hood finite elements, as described in Section~\ref{sec:aposteriori}, and choose $\bg_h$ as the Lagrange interpolant. In this case, the compatibility condition is automatically satisfied. Otherwise, one may use the modification proposed in \cite{duran2020} or other regularization strategies such as those described in \cite{apel2026}.

We next discuss how the discrete solution is computed on each fixed mesh. The numerical solution of the Navier--Stokes problem is typically obtained through an iterative linearization procedure. In our numerical experiments, we employ a Newton scheme, whose convergence has been studied for both the homogeneous \cite{girault1986,karakashian1982} and non-homogeneous cases \cite{gunzburger1983,wang2016}. All numerical experiments were carried out using the finite element library \textit{NGSolve} \cite{ngsolve}.

Since we do not know the exact solution and that our numerical solution is the result of an iterative method, the $L^4(\O)$-error is calculated as the difference between the limits solutions obtained in two consecutive refinements. 
In fact, let us note as $\bu_{h_{j}}^{\star}$ and  $\bu_{h_{j+1}}^{\star}$ the limits of the sequence $\bu^n_h$ for a mesh $j$ and its refinement $j+1$, where $\bu_h^n$ is the sequence given by Newton's method. Then, we define successive errors as,
$$e_{L^4}(\bu)= \| \bu^{\star}_{h_{j}} - \bu^{\star}_{h_{j+1}} \|_{L^4(\O)}$$

As usual (see, for instance, \cite{verfurth1996}), the adaptive procedure consists of computing the numerical solution on the mesh $\matT_h$, marking the elements for refinement according to a given strategy, and constructing a new mesh $\matT'_h$, which is a refinement of $\matT_h$. We consider the \textit{average strategy} as a marketing strategy. Given a parameter $\theta\in(0,1]$, the average strategy marks all elements $T'\in\mathcal{T}_h$ such that
$$
\eta_{T'}
\ge
\theta
\frac{1}{\#\matT_h}\sum_{T\in\mathcal{T}_h}\eta_T.
$$

Table \ref{tab:uniform} reports the results obtained after five levels of uniform refinement and we compute the experimental order of convergence (eoc) with respect to the number of degrees of freedom (DOFs) and Figure \ref{fig:uniform_ref} show the refinement for the uniform scheme. Since $\text{DOFs} \sim h^{-2}$ for quasi-uniform meshes in two dimensions, a convergence rate of order $h^\alpha$ corresponds to a rate of order DOFs $-\alpha/2$. Therefore, the rates observed with respect to the number of degrees of freedom are expected to be half of those with respect to $h$. The numerical results indicate a convergence rate close to $-1/4$ with respect to the number of degrees of freedom. Equivalently, this yields a rate close to  $1/2$ with respect to the mesh size $h$,  which is consistent with the theoretical and numerical findings of \cite{ArMe} since cavity boundary data belongs to $H^{
1/2-\eps}(\Gamma)$ for every $\eps >0$.

\begin{table}
\centering
\caption{Uniform scheme.}
\label{tab:uniform}   
\begin{tabular}{lll}
\hline\noalign{\smallskip}
DOFs & $e_{L^4}$ & eoc  \\
\noalign{\smallskip}\hline\noalign{\smallskip}
288 &  2.3620e-01 & -- \\
1040 & 1.7331e-01 & -0.2411\\
3948 & 1.2262e-01 & -0.2505\\
15380 & 8.6658e-02 & -0.2528\\
60708 & 6.1273e-02 & -0.2532 \\
\noalign{\smallskip}\hline
\end{tabular}
\end{table}

\begin{figure}[!ht]
    \centering
    \begin{subfigure}{0.48\textwidth}
        \centering
        \includegraphics[width=1\linewidth]{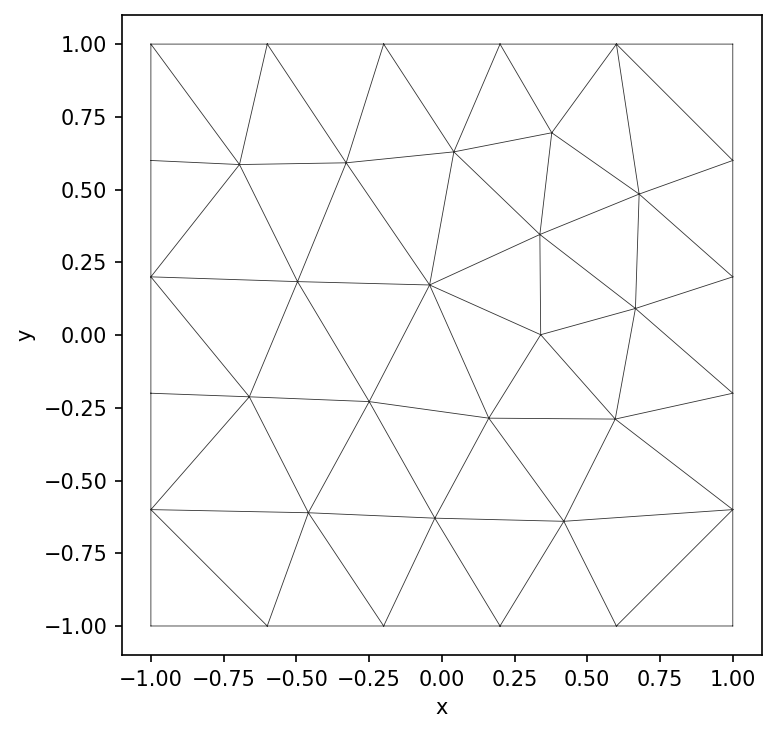}
        \label{fig:initial_mesh_uniform}
    \end{subfigure}
    \hfill
    \begin{subfigure}{0.48\textwidth}
        \centering
        \includegraphics[width=1\linewidth]{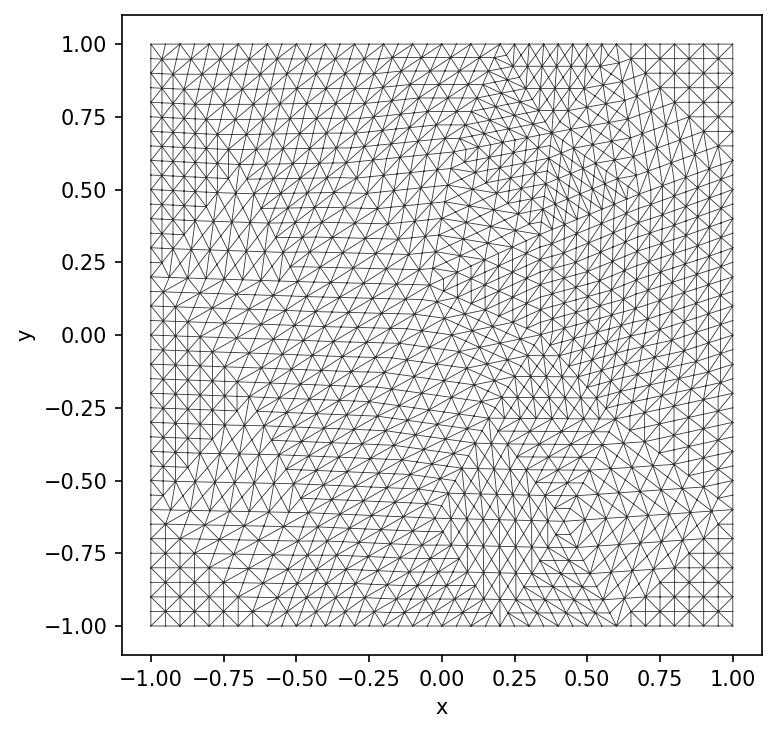}
        \label{fig:final_mesh_uniform}
    \end{subfigure}
    \caption{Meshes generated by the uniform scheme}
    \label{fig:uniform_ref}
\end{figure}

Table~\ref{tab:avg_strategy} exhibits the experimental order of convergence of both the error and the estimator with respect to the number of degrees of freedom and compute he efficiency index given by $$\text{eff} = \frac{\eta}{e_{L^4}}.$$ 
Figure \ref{fig:convergence} displays the log-log plot of the error and estimator versus the number of degrees of freedom, while Figure~\ref{fig:refinamiento_avg} shows that the mesh refinement is concentrated near the discontinuities of the boundary datum $\bg$. 

We observe that the proposed estimator not only concentrates the mesh refinement around the singularities, as expected, but also improves the convergence rate. Moreover, a smaller error is achieved with a significantly lower number of degrees of freedom than that required by uniform refinement.

\begin{table}
\centering
\caption{Adaptive scheme for the Hood-Taylor method using the average strategy and $\theta = 0.6$}
\label{tab:avg_strategy}   
\begin{tabular}{llllll}
\hline\noalign{\smallskip}
DOFs & $e_{L^4}$ & eoc ($e_{L^4}$) & $\eta$ & eoc ($\eta$) & eff \\
\noalign{\smallskip}\hline\noalign{\smallskip}
288 & 2.3610e-01 & -- & 8.3308e-01 & -- & 3.5285 \\
404 & 1.8456e-01 & -0.7276 & 8.3160e-01 & -0.0053 & 4.5057 \\
547 & 1.3313e-01 & -0.8900 & 5.4173e-01 & -0.6580 & 4.0693 \\
772 & 9.7867e-02 & -0.9093 & 3.8943e-01 & -0.8276 & 3.9792 \\
1134 & 6.8207e-02 & -0.9199 & 2.7312e-01 & -0.8824 & 4.0043 \\
1564 & 5.1723e-02 & -0.9146 & 1.9498e-01 & -0.9188 & 3.7697 \\
2289 & 3.8435e-02 & -0.8964 & 1.3661e-01 & -0.9337 & 3.5543 \\
2952 & 2.6779e-02 & -0.9209 & 9.6208e-02 & -0.9663 & 3.5927 \\
3772 & 2.1735e-02 & -0.9295 & 6.8428e-02 & -1.0015 & 3.1482 \\
5144 & 1.4988e-02 & -0.9475 & 4.8089e-02 & -1.0278 & 3.2085 \\
\noalign{\smallskip}\hline
\end{tabular}
\end{table}

\vspace{1cm}

\begin{figure}[!ht]
    \centering
        \includegraphics[width=0.57\linewidth]{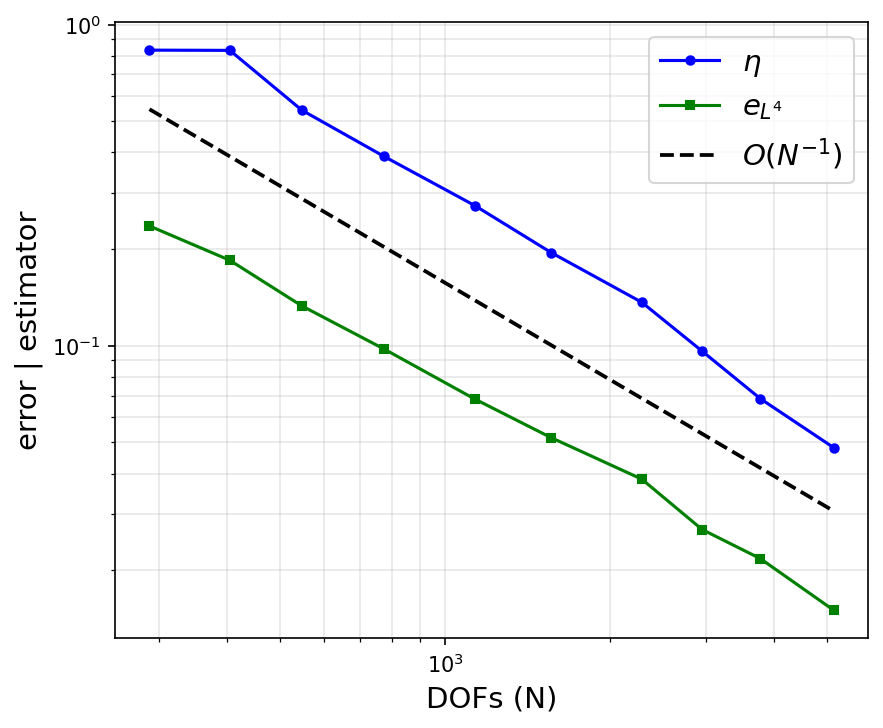}
           \caption{Log--log plot of the error and estimator versus the number of degrees of freedom }
    \label{fig:convergence}
\end{figure}

\begin{figure}[!ht]
    \centering
    \begin{subfigure}{0.475\textwidth}
        \centering
        \includegraphics[width=\linewidth]{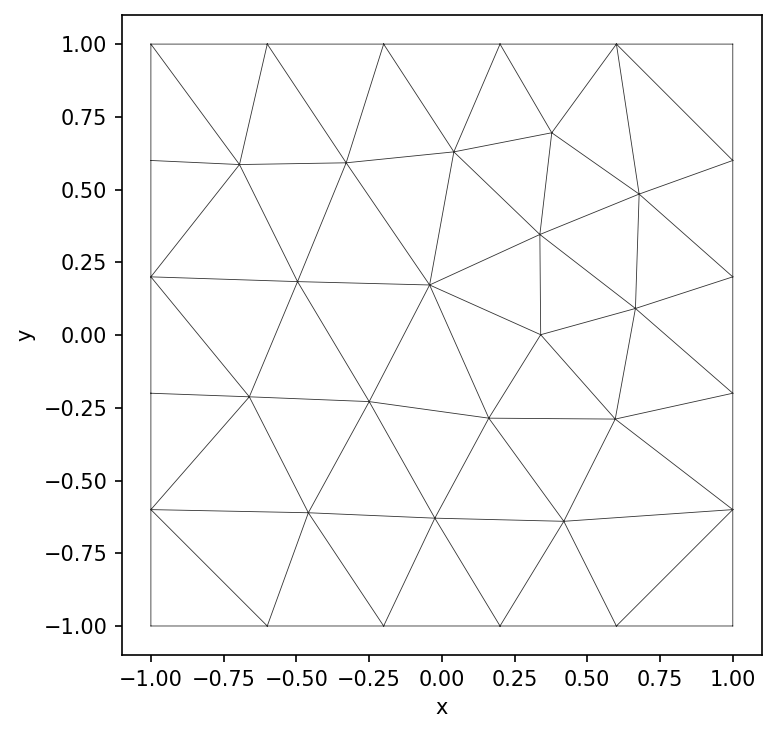}
        \label{fig:g1}
    \end{subfigure}
    \hfill
    \begin{subfigure}{0.475\textwidth}
        \centering
        \includegraphics[width=\linewidth]{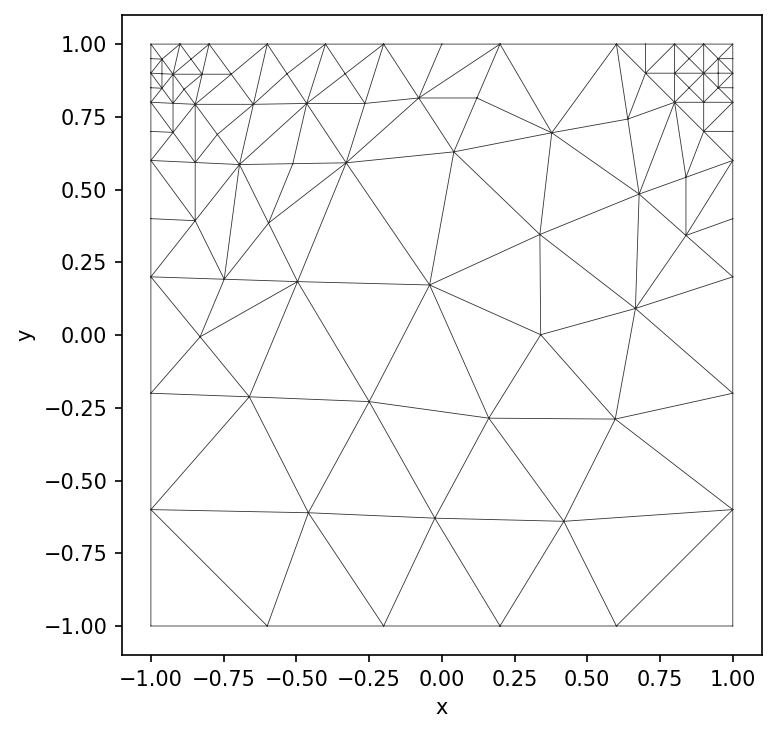}
        \label{fig:g2}
    \end{subfigure}

    \vspace{0.5cm}

    \begin{subfigure}{0.475\textwidth}
        \centering
        \includegraphics[width=\linewidth]{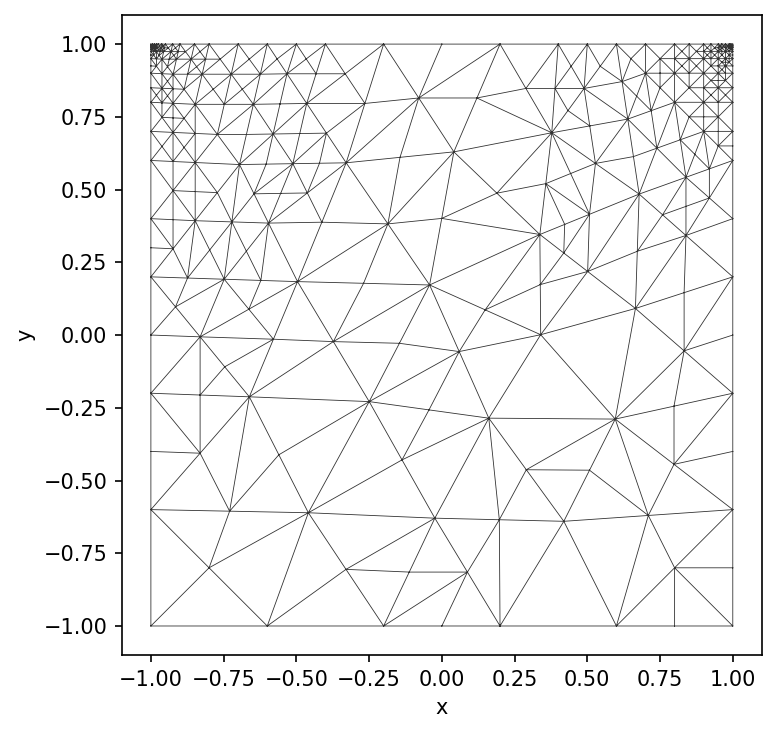}
        \label{fig:g3}
    \end{subfigure}
    \hfill
    \begin{subfigure}{0.475\textwidth}
        \centering
        \includegraphics[width=\linewidth]{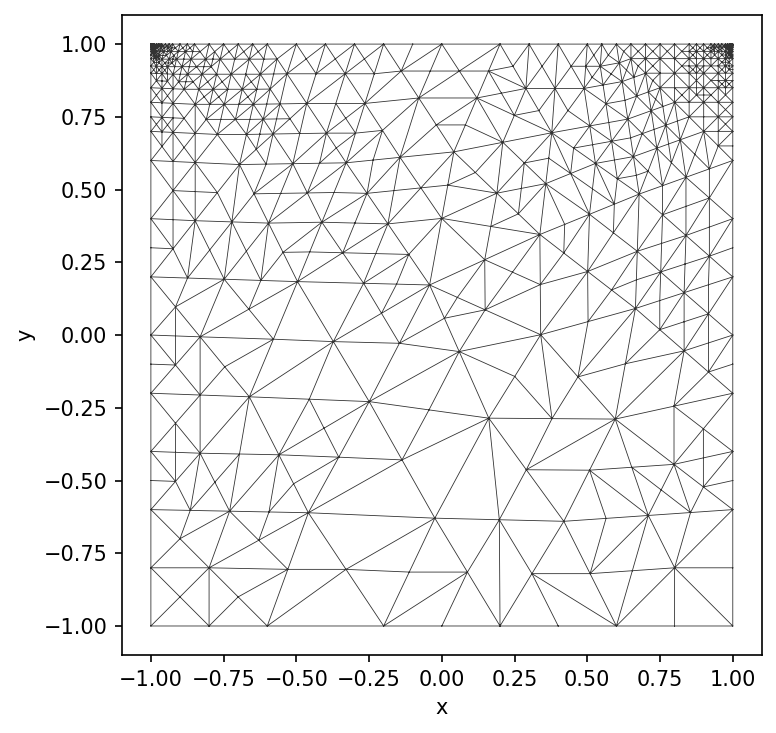}
        \label{fig:g4}
    \end{subfigure}

    \caption{Sequence of meshes generated by the adaptive procedure using the average strategy and the local error indicators $\eta_T$ with parameter $\theta=0.6$. The initial mesh and the meshes obtained after 3, 7, and 10 adaptive iterations are shown.}
    \label{fig:refinamiento_avg}
\end{figure}

\clearpage
\bibliographystyle{plain}
\bibliography{reference}
\end{document}

%% file: paquetes.tex
\usepackage{graphicx}
\usepackage{tikz}
\usetikzlibrary{cd,matrix,arrows,decorations.pathmorphing}
\usepackage{
  amsmath,
  amsthm,
  amssymb,
  euscript,
  enumerate, 
  url,
  verbatim,
  calc,
}
\usepackage{geometry}
\usepackage{algorithm}
\usepackage{algpseudocode}
\usepackage{subcaption}
\usepackage[colorinlistoftodos]{todonotes}
\usepackage{amsmath,amssymb}
\usepackage{xcolor}
\usepackage{authblk}
\usepackage{subcaption}

%% file: comandos.tex
\def\matR{\mathbb{R}}
\def\Rey{\mathrm{Re}}
\def\O{\Omega}
\def\bu{\textit{\textbf{u}}}

\def\bv{\textit{\textbf{v}}}
\def\bpsi{\boldsymbol{\Psi}}
\def\bz{\mathbf{z}}
\def\bw{\textit{\textbf{w}}}
\def\bf{\mathbf{f}}
\def\bg{\mathbf{g}}

\def\bxi{\mathbf{\xi}}
\def\bn{\mathbf{n}}

\def\bE{\mathbf{E}}
\def\bF{\mathbf{F}}
\def\bG{\mathbf{G}}
\def\bI{\mathbf{I}}

\def\bL{\mathbf{L}}
\def\bH{\mathbf{H}}
\def\bW{\mathbf{W}}
\def\bV{\mathbf{V}}

\def\bJ{\mathbf{J}}
\def\div{\textup{div} \ }

\def\matT{\mathcal{T}}
\def\matE{\mathcal{E}}

\def\matC{\mathcal{C}}

\def\eps{\varepsilon}

\theoremstyle{plain}
\newtheorem{theorem}{Theorem}[section]
\newtheorem{corollary}{Corollary}[section]
\newtheorem{lemma}[theorem]{Lemma}
\newtheorem{prop}[theorem]{Proposition}
\theoremstyle{definition}
\newtheorem{remark}[theorem]{Remark}

\newtheorem{assumption}{Assumption}